\documentclass[12pt,reqno]{amsart}
\usepackage{amsmath, amssymb, verbatim, comment, color, mathtools, accents}
\usepackage[margin=3cm]{geometry}

\makeatletter
\def\widebar{\accentset{{\cc@style\underline{\mskip10mu}}}}
\makeatother

\newcommand{\R}{\mathbb{R}}

\renewcommand{\a}{\alpha}
\renewcommand{\b}{\beta}
\newcommand{\g}{\gamma}

\renewcommand{\t}{\theta}

\newcommand{\na}{\nabla}

\newcommand{\nap}{\nabla^{\perp}}

\DeclareMathOperator{\Ric}{Ric}

\DeclareMathOperator{\Hess}{Hess}

\DeclareMathOperator{\dvol}{dvol}
\DeclareMathOperator{\supp}{supp}

\newcommand{\lang}{\langle}
\newcommand{\rang}{\rangle}

\numberwithin{equation}{section}       

\theoremstyle{plain}
\newtheorem{theorem}{Theorem}[section]
\newtheorem{proposition}[theorem]{Proposition}
\newtheorem{lemma}[theorem]{Lemma}
\newtheorem{corollary}[theorem]{Corollary}

\theoremstyle{definition}

\theoremstyle{remark}
\newtheorem{remark}[theorem]{Remark}
\newtheorem*{acknowledgement}{\bf{Acknowledgment}} 

\title[Remarks on stability of translating solitons]{Remarks on topology of stable translating solitons}\date{\today}

\author[K. Kunikawa]{Keita Kunikawa$^*$}
\author[S. Saito]{Shunsuke Saito$^{\dagger}$}
\address{Advanced Institute for Materials Research \\
Tohoku University \\
Katahira \\
Aoba-ku \\
Sendai 980-8578 \\
Japan}
\email{keita.kunikawa.e2@tohoku.ac.jp \!$^*$}
\email{shunsuke.saito.c4@tohoku.ac.jp \!$^{\dagger}$}

\begin{document}

\begin{abstract}
We show that any complete $f$-stable translating soliton $M$ admits no codimension one cycle which does not disconnect $M$. As a corollary, it follows that any two dimensional complete $f$-stable translating soliton has genus zero. 
\end{abstract}

\maketitle

\section{Introduction}\label{intro}
A \textit{translating soliton} (\textit{translator} for short) is an oriented, connected smoothly immersed hypersurface $x:M^n \to \R^{n+1}$ on which its mean curvature vector $H$ satisfies  
\begin{align*} 
H=T^\perp, 
\end{align*}
where $T \in \R^{n+1}$ is a fixed unit length constant vector and $T^\perp$ denotes its normal projection onto the normal bundle $T^\perp M$. Translators are known as type II singularity model of the mean curvature flow and hence classification of translators gives us a better understanding of singularities. In \cite{MSS15}, Mart\'in, Savas-Halilaj and Smoczyk studied topological aspects of translators, and they asked whether there exists a complete translator which has finite genus when $\dim M = 2$. Later, Smith \cite{Smi15} gave an answer to this question. Adopting desingularisation technique, he constructed new complete embedded translators which have arbitrary finite genus and three ends. Note that complete embedded translators of infinite genus were already known by Nguyen in \cite{Ngu09}, \cite{Ngu13} before the works \cite{MSS15}, \cite{Smi15}.   

On the other hand, translators are critical points of the weighted volume functional $\mathcal{A}_T(M)$ with respect to the weighted measure $e^{\lang x, T \rang}\dvol_M$. In the following, we set $f(x):=-\lang x, T\rang\in C^\infty(M)$ and $\dvol_{f}:=e^{-f}\dvol_M$. A translator $x:M^n \to \R^{n+1}$ is called \textit{$f$-stable} if the second derivative of $\mathcal{A}_T$ is nonnegative for all compactly supported normal variations of $x$.  
Shahriyari \cite{Sha15} studied the $f$-stability of translators and he showed that any translating graph $M^2\subset \R^3$ must be $f$-stable. Later, this fact was generalized to higher dimensions by Xin \cite{Xin15}. Clearly, every complete graph has no genus (when $\dim M=2$) and only one end. Hence, roughly speaking, simple topology implies $f$-stability. Our interest is the converse, that is, topological properties of translators under $f$-stability. In this direction, Impera and Rimoldi showed the following. 
\begin{proposition}[Impera-Rimoldi \cite{IR17}]\label{1end}
For $n\geq 2$, any complete $f$-stable translator $x:M^n \to \R^{n+1}$ has at most one end. 
\end{proposition}
Their result indicates that $f$-stable translators must be topologically simple (at least in the sense of ends). As an immediate consequence, all examples constructed by Smith \cite{Smi15} can not be $f$-stable since they have three ends. 

In the present paper we also confirm topological simplicity of $f$-stable translators from a view point of a relation between weighted $L^2$ harmonic $1$-forms and codimension one cycles (the details are explained in a later section). 
\begin{theorem} \label{main}
For $n\geq 1$, any complete $f$-stable translator $x:M^n \to \R^{n+1}$ admits no codimension one cycle which does not disconnect $M$. 
\end{theorem}
Especially in the surface case ($n=2$), Theorem \ref{main} implies the following: 	
\begin{corollary}
Any two dimensional complete $f$-stable translator $x:M^2 \to \R^3$ has genus zero.  
\end{corollary}

\begin{remark}
Recently, Impera and Rimoldi \cite{IR18} showed a quantitative relation between the $f$-index and the genus of translators. 
\end{remark}

\begin{acknowledgement}
	The first author would like to thank Miyuki Koiso and Reiko Miyaoka for their valuable comments and discussion on stability of translators. The second author was supported by Structural Materials for Innovation Strategic Innovation Promotion Program D72.
\end{acknowledgement}

\section{Preliminaries} 
Throughout the paper, we assume that $x:M^n \to \R^{n+1}$ is a codimension one, smooth complete translator with the induced metric $g:=x^*\lang \cdot, \cdot\rang$ and the second fundamental form $A$. Also we assume that $M$ is oriented and connected. 
\subsection{Stability of translators} 
As mentioned in the introduction, a translator is defined by the equation $H=T^\perp$ for some unit length constant vector $T\in \R^{n+1}$. By the usual first variation formula, it is not difficult to see that for a hypersurface being a translator is equivalent to being a critical point of the weighted volume functional $\mathcal{A}_T(M)$. Then the \textit{$f$-stability} (or it is also called \textit{$L$-stability}) of a translator is defined by 
\begin{align*} 
\frac{d^2}{dt^2}\Big|_{t=0}\mathcal{A}_T(x_t(M))\geq 0 
\end{align*}
for any compactly supported normal variation $x_t:M \to \R^{n+1}$ with $x_0=x$. Let $V\in T^\perp M$ be a variation vector field of such a compactly supported normal variation $x_t$. Then $f$-stability for translators can be written as 
\begin{align} \label{stab}
\int_M|\nap V|^2-|A|^2|V|^2\dvol_f\geq 0,  
\end{align}
where $\nap$ denotes the normal connection on $T^\perp M$. See \cite{IR15}, \cite{IR17} or \cite{Sha15} for more details. 

\subsection{Weighted Laplacian} Let $A^p(M)$ be the space of smooth $p$-forms on $M$ and $A^p_c(M)$ be the space of compactly supported $p$-forms on $M$. For $\a, \b\in A^p(M)$, let 
\begin{align*}
(\a, \b)_g:=\star_g(\a\wedge \star_g \b) 
\end{align*}
be a standard pointwise inner product on forms, where $\star_g$ is the Hodge star operator with respect to $g$. Then for every $\a, \b\in A^p_c(M)$, we define the $L^2_f$ inner product by 
\begin{align*}
\lang \a, \b \rang_{L^2_f}:=\int_M(\a, \b)_g \dvol_f.   
\end{align*}
The completion of $A^p_c(M)$ with respect to the norm $||\cdot||^2_{L^2_f}=\lang \cdot, \cdot \rang_{L^2_f}$ is denoted by $L^2_fA^p(M)$. Then the exterior derivative $d:A^p_c(M) \to A^{p+1}_c(M)$ is densely defined operator on $L^2_fA^p(M)$, and it has the formal adjoint $\delta_f$ of $d$ with respect to the inner product $\lang \cdot, \cdot\rang_{L^2_ f}$ on $A^{p+1}_c(M)$. Now we define the \textit{weighted Laplacian} acting on $A^p_c(M)$ by $d\delta_f+\delta_fd$ which is formally self-adjoint and densely defined on $L^2_fA^p(M)$. Note that $d\delta_f+\delta_fd$ has a unique closed extension which is denoted by $\Delta_f$ and it is self-adjoint. A $p$-form $\a \in A^p(M)$ is called \textit{$L^2_f$ $f$-harmonic} if $\a\in L^2_fA^p(M)$ and satisfies $\Delta_f\a=0$. The space of all such $p$-forms is denoted by 
\begin{align*}
\mathcal{H}^p_f(M):=\{\a\in L^2_fA^p(M)\; |\; \Delta_f\a=0\}. 
\end{align*} 
It is known that $\mathcal{H}^p_f(M)\subset A^p(M)$, namely, $L^2_f$ $f$-harmonic forms are smooth. See \cite{Bue99} for details.

\section{Weighted harmonic forms and topology} 
In this section, we assume that $M$ is an oriented connected complete  Riemannian manifold, but not necessarily be translator. 

The $p$-th de Rham cohomology group with compact support on $M$ is defined by 
\begin{align*}
H^p_c(M)=\frac{\{\a \in A^p_c(M)\;|\; d\a=0\}}{dA^{p-1}_c(M)}. 
\end{align*}
We also consider the space of $L^2_f$ closed forms: 
\begin{align*}
L^2_fZ^p(M):=\{\a \in L^2_fA^p(M)\;|\; d\a=0\}, 
\end{align*}
where it is understood that $d\a=0$ holds weakly. Then we define the \textit{reduced $L^2_f$ cohomology} by 
\begin{align*}
L^2_fH^p(M):=\frac{L^2_fZ^p(M)}{\overline{dA^{p-1}_c(M)}}, 
\end{align*}
where $\overline{dA^{p-1}_c(M)}$ is the closure of $dA^{p-1}_c(M)$ with respect to the norm $||\cdot||_{L^2_f}$. The reduced $L^2_f$ cohomology on a complete Riemannian manifold is isomorphic to the space of $L^2_f$ $f$-harmonic forms (see \cite{Bue99}): 
\begin{align*}
\mathcal{H}^p_f(M)\cong L^2_fH^p(M). 
\end{align*}

\subsection{Weighted $L^2$ harmonic forms and weighted Sobolev inequality}
First, we consider the following approximation theorem (see \cite{Car07}). 
\begin{lemma} \label{appro}
Let $\a \in L^2_fZ^1(M)\cap A^1(M)$ and suppose that $\a$ is zero in $L^2_fH^1(M)$, namely, there is a sequence of smooth functions $(u_i)_i$ in $C^\infty_c(M)$ such that 
\begin{align*}
du_i \to \a \quad (i \to \infty)\quad \textup{ in $L^2_f$}.  
\end{align*}
Then there exists $u\in C^\infty(M)$ such that $\a =du$. 
\end{lemma}
\begin{proof}
By the Poincar\'e duality, it is enough to show that 
\begin{align*}
\int_M\a \wedge \eta =0 
\end{align*}
for every closed form $\eta\in A^{n-1}_c(M)$. For such $\eta$, define a $1$-form $\t\in A^1(M)$ by $\t:=(-1)^{n-1}e^f\star_g\eta$. Then 
\begin{align*}
\int_M\a \wedge \eta&=\int_M\a \wedge \star_g(e^{-f}\t)=\int_M(\a, \t)_g\dvol_f\\
&=\lim_{i \to \infty}\int_M(du_i, \t)_g\dvol_f=\lim_{i \to \infty}\int_M du_i\wedge \eta
=\lim_{i\to \infty}\int_Md(u_i\eta)=0. 
\end{align*}
\end{proof}

If the weighted Sobolev inequality holds on $M$, we obtain the following result as an application of the approximation theorem (Lemma \ref{appro}). 
\begin{proposition}\label{sobolev}
Let $M$ be an $n$-dimensional complete Riemannian manifold. Suppose that $M$ supports the $L^1_f$ Sobolev inequality, that is, there exists a constant $C>0$ such that for every compactly supported nonnegative function $h\in C^\infty_c(M)$, 
\begin{align*}
||h||_{L^{\frac{n+1}{n}}_f}\leq C||dh||_{L^1_f}. 
\end{align*}
Then the natural map 
\begin{align*}
H^1_c(M) \to L^2_fH^1(M)\cong \mathcal{H}^1_f(M)
\end{align*}
is injective. 
\end{proposition}
\begin{proof} 
It is known by Impera-Rimoldi \cite{IR17} that every end has infinite $f$-volume as a consequence of the $L^1_f$ Sobolev inequality. Also applying the $L^1_f$ Sobolev inequality to the function $h=v^{\frac{2n}{n-1}}, v\in C^\infty_c(M)$, we obtain the validity of the $L^2_f$ Sobolev inequality 
\begin{align*}
||v||_{L^{\frac{2\nu}{\nu-2}}_f}\leq \mu(n)||dv||_{L^2_f}, 
\end{align*}
where $\nu=n+1>2$ and $\mu(n)>0$ is a constant depending only on $n$. 

Now, take a closed $1$-form $\a \in A^1_c(M)$ which is zero in $L^2_fH^1(M)$. Then there exists a sequence $(u_i)_i$ in $C^\infty_c(M)$ such that $du_i \to \a$ in $L^2_f$. Hence by Lemma \ref{appro}, there exists a function $u\in C^\infty(M)$ such that $\a = du$, and by the $L^2_f$ Sobolev inequality, $u_i \to u$ in $L^{\frac{2\nu}{\nu-2}}_f$. Since $du= \a = 0$ in $M\setminus \supp \a$, $u$ is constant on each connected component of $M\setminus \supp \a$. However, by the $L^2_f$ Sobolev inequality, $u$ must be zero on each unbounded connected component of $M$ since each end of $M$ has infinite $f$-volume. Hence $u$ has compact support and $\a$ is zero in $H^1_c(M)$.  
\end{proof} 

\subsection{Weighted $L^2$ harmonic forms and cycles}
Let $Z$ be a $p$-dimensional oriented, connected, compact manifold without boundary and $\phi:Z \to M^n$ be an embedding. The pair $(Z, \phi)$ is called a \textit{$p$-cycle} on $M$. We say that an $(n-1)$-cycle (or \textit{codimension one cycle}) $Z$ does not disconnect $M$ if $M\setminus \phi(Z)$ is connected. We consider a relation between $H^1_c(M)$ and codimension one cycles. 
\begin{proposition} \label{car}
Assume that there exists a codimension one cycle $Z$ which does not disconnect $M$. Then there exists a closed $1$-form $\a \in A^1_c(M)$ and a $1$-cycle $\g$ on $M$ such that 
\begin{align*}
\int_\g \a = 1. 
\end{align*}
Therefore $H^1_c(M)\neq \{0\}$. 
\end{proposition} 
\begin{proof}
This proof is deeply inspired by the one contained in the lecture note by Carron \cite{Car07}. Consider an embedding $F:Z \times [-1, 1] \to M$ with $F_p(0)=F(p, 0)=\phi(p)$ for any $p\in Z$, and set $N_Z=F(Z \times [-1, 1])$. Then $M\setminus N_Z$ is connected since $M\setminus \phi(Z)$ is connected by assumption. Now define a function $\rho:Z \times [-1, 1] \to \R$ by 
\begin{align*}
\rho(p, t)=\rho(t)=\begin{cases}
1 \quad (1/2\leq t\leq 1), \\
0 \quad (-1\leq t\leq -1/2). 
\end{cases}
\end{align*}
Clearly, $d\rho$ has a compact support on $Z\times (-1, 1)$. Hence we can extend $d(F^{-1})^*\rho$ to a closed $1$-form $\a$ defined on whole $M$ with compact support in $N_Z$. 

For a fixed point $p_0\in Z$, we can take a continuous curve $c:[0, 1] \to M\setminus N_Z$ joining $F(p_0, 1)$ and $F(p_0, -1)$ since $M\setminus N_Z$ is connected. Define a closed curve (a $1$-cycle) $\g$ on $M$ by 
\begin{align*}
\g(t)=\begin{cases}
F_{p_0}(t):=F(p_0, t)\quad (-1\leq t\leq 1), \\
c(t-1) \quad (1\leq t\leq 2). 
\end{cases} 
\end{align*}
It is not difficult to see that 
\begin{align*}
\int_{\g}\a=1, 
\end{align*}
hence $\a$ is not exact. 
\end{proof}

Combining Proposition \ref{sobolev} and Proposition \ref{car}, we immediately conclude the following. 
\begin{corollary}
Assume that $M$ supports the $L^1_f$ Sobolev inequality in Proposition \ref{sobolev} and admits a codimension one cycle which does not disconnect $M$, then $\mathcal{H}^1_f(M)\neq \{0\}$. 
\end{corollary}

In \cite{IR17}, Impera and Rimoldi showed that every complete translator $x:M^n \to \R^{n+1}$ supports the $L^1_f$ Sobolev inequality when $n\geq 2$. Hence, in particular, we conclude the following. 
\begin{corollary} \label{tracyc}
If a complete translator $x:M^n \to \R^{n+1}$ admits a codimension one cycle which does not disconnect $M$, then $\mathcal{H}^1_f(M)\neq \{0\}$. 
\end{corollary}

\subsection{Another method} 
We may adopt another proof of Corollary \ref{tracyc} without the Sobolev inequality. The following proposition is essentially due to Gaffney \cite{Gaf54} where he used the heat equation method. Most of his method can be applicable to our weighted Laplacian case straightforwardly since the method is based on general theory for self-adjoint operators. The only thing we need to do is to modify the proof such as Lemma \ref{appro} when we use the Poincar\'e duality. 
\begin{proposition} \label{gaf}
For any closed $p$-form $\a \in A^p_c(M)$ and any $p$-cycle $Z$ on $M$, 
\begin{align*}
\int_Z\a=\int_Z \mathbb{H}_f\a, 
\end{align*}
where $\mathbb{H}_f\a$ denotes the $f$-harmonic part of $\a$. 
\end{proposition} 

Combining Proposition \ref{car} and Proposition \ref{gaf} we recover the result by Dodziuk \cite{Dod82} in our weighted case. Note that $M$ does not need to be a translator in the following result. 
\begin{corollary} \label{dodz}
Assume that $M$ admits a codimension one cycle $Z$ which does not disconnect $M$, then $\mathcal{H}^1_f(M)\neq \{0\}$. 
\end{corollary}

\section{$f$-stability and weighted $L^2$ harmonic $1$-forms} 
In this section, we prove the main theorem: 
\begin{theorem} \label{mainth}
Any complete $f$-stable translator $x:M^n \to \R^{n+1}$ admits no codimension one cycle which does not disconnect $M$. 
\end{theorem} 
\begin{proof}
This proof is inspired by the works of Palmer \cite{Pal91} or Miyaoka \cite{Miy93} for stable minimal hypersurfaces. 

We reason by contradiction. Assume that there exists a codimension one cycle which does not disconnect $M$. Then by Corollary \ref{dodz} (or Corollary \ref{tracyc}), we can take a non-trivial $L^2_f$ $f$-harmonic $1$-form $\a\in \mathcal{H}^1_f(M)$. Let $\xi\in TM$ be the dual vector field of $\a$ with respect to $g$.  Note that the weighted Laplacian can be computed as $\Delta_f=\Delta+L_{\na f}$, where $L_{\na f}$ denotes the Lie derivative. Using Weitzenb\"ock formula and $\Delta_f\a=0$, we have 
\begin{align*}
2|\xi|\Delta_f|\xi|-2|\na |\xi||^2=\Delta_f|\xi|^2=-2\Ric_f(\xi,\xi)-2|\na \xi|^2, 
\end{align*}
where $\Ric_f:=\Ric+\Hess_f$ is Bakry-\'Emery Ricci tensor on $M$. 
Set  
\begin{align*}
K(\xi)&:=|\na \xi|^2-|\na|\xi||^2, \\
W(\xi)&:=|A|^2|\xi|^2-\lang S(\xi), S(\xi)\rang, 
\end{align*}
where $S$ is the shape operator of the hypersurface. 
Note that $K(\xi)\geq 0$ by Kato's inequality. Also $W(\xi)\geq 0$ in general. Using the Gauss equation, we compute 
\begin{align} \label{kato}
|\xi|\Delta_f|\xi|&= -\Ric_f(\xi, \xi)-K(\xi)\\
&=\lang S(\xi), S(\xi)\rang-K(\xi) \nonumber \\
&=|A|^2|\xi|^2-W(\xi)-K(\xi).  \nonumber
\end{align}

Now we take a cutoff function 
$\eta\in C^\infty(M)$ which has the property 
\begin{enumerate}
\item[(a)] $0\leq \eta \leq 1$, 
\item[(b)] $\eta\equiv 1$ on $B_{\frac{R}{2}}$, \quad $\eta\equiv 0$ outside $B_R$, 
\item[(c)] $|\na \eta|^2\leq \displaystyle\frac{C}{R^2}$, \quad $C$ is a constant independent of $R$,  
\end{enumerate}
where $B_R$ is a geodesic ball on $M$ with radius $R>0$. Consider a variation vector field as $V=\eta|\xi|\nu$, where $\nu$ is a unit normal of the hypersurface. From the stability inequality \eqref{stab}, the equality \eqref{kato} and Stokes' theorem $\lang \Delta_fu, v \rang_{L^2_f}=\lang du, dv\rang_{L^2_f}$ for $u, v\in C^\infty_c(M)$, we compute 
\begin{align*}
0&\leq \int_{M}|\nap V|^2-|A|^2|V|^2\dvol_f\\
&=\int_M |\na(\eta|\xi|)|^2-|A|^2\eta^2|\xi|^2\dvol_f\\
&=\int_M \eta|\xi|\Delta_f(\eta|\xi|)-|A|^2\eta^2|\xi|^2\dvol_f\\
&=\int_M \eta|\xi|(\eta\Delta_f|\xi|+|\xi|\Delta_f \eta-2\lang \na \eta, \na |\xi|\rang)-|A|^2\eta^2|\xi|^2\dvol_f\\
&=\int_M \eta^2 (|\xi|\Delta_f|\xi|-|A|^2|\xi|^2)-2\eta|\xi|\lang \na |\xi|,\na \eta\rang+(|\xi|^2\eta)\Delta_f \eta\dvol_f\\
&= \int_M -\eta^2 (W(\xi)+K(\xi))+|\xi|^2|\na \eta|^2\dvol_f\\
&\leq \int_{B_{\frac{R}{2}}} -(W(\xi)+K(\xi))\dvol_f + \frac{C}{R^2}\int_{B_R}|\xi|^2\dvol_f. 
\end{align*}
Letting $R \to \infty$ and using $\a \in L^2_fA^1(M)$, we conclude that $W(\xi)\equiv 0$ and $K(\xi)\equiv 0$ on whole $M$. Since $\xi$ is not identically zero, we choose a point $p\in M$ where $|\xi(p)|\neq 0$. We may take a sufficiently small neighborhood $U$ of $p$ on which $|\xi|>0$ holds. Let $\{e_1, \cdots, e_n\}$ be an orthonormal frame on $U$ such that $e_1=\xi/|\xi|$. Then $W(\xi)\equiv 0$ implies 
\begin{align*}
|A|^2=\lang S(e_1), S(e_1)\rang=\lang A(e_1, e_i), A(e_1, e_i) \rang=a_{11}^2+\cdots +a_{1n}^2,  
\end{align*}
where $a_{ij}=A(e_i, e_j)$. Since $A$ is symmetric, $a_{ij}=0$ except for $a_{11}$. It follows that the scalar curvature of $M$ must be zero. However it is known that a zero scalar curvature translator is either a hyperplane or a grim reaper cylinder by Mart\'in, Savas-Halilaj and Smoczyk (Corollary 2.1 in \cite{MSS15}). This contradicts our assumption since both of them do not have codimension one cycle which does not disconnect $M$ by the Jordan-Brouwer separation theorem. 
\end{proof}
\begin{remark}
Our proof of the main theorem may be used to prove Proposition \ref{1end} in a slightly different way. In fact, if a complete Riemannian manifold $M$ has at least two ends, then $H^1_c(M)\neq \{0\}$ (see the lecture note by Carron \cite{Car07} for the proof). In particular, for a translator $x:M^n \to \R^{n+1}$, $\mathcal{H}^1_f(M)\neq \{0\}$ by Proposition \ref{sobolev}. However, as in Theorem \ref{mainth}, this is in contradiction with the $f$-stability assumption.  
\end{remark}

\end{document}